\documentclass[letterpaper, 10 pt, conference]{ieeeconf}  
\IEEEoverridecommandlockouts                              

\usepackage{amsfonts}
\usepackage{amsmath}
\usepackage{graphicx}
\usepackage{booktabs}
\usepackage{xcolor}
\usepackage{mathtools}
\usepackage{slashbox}  
\newcommand{\R}{\mathbb{R}}

\newcommand{\bmat}[1]{\begin{bmatrix}#1\end{bmatrix}}
\newcommand{\norm}[1]{\left\lVert{#1}\right\rVert}

\newcommand{\ip}[2]{\left\langle #1, #2 \right\rangle}

\newcommand{\mcl}[1]{\mathcal{ #1}}
\newcommand{\mbf}[1]{\mathbf{ #1}}

\newtheorem{thm}{Theorem}

\newtheorem{lem}[thm]{Lemma}

\newcommand{\fourpi}[4]{\ensuremath{\mcl{P}{\tiny\bmat{#1,& \hspace{-3mm}#2 \\ #3,& \hspace{-3mm} \left\{#4\right\} }}}}

\begin{document}
	
	\title{\LARGE \bf {$H_\infty$-Optimal Observer Design for Linear Systems with Delays in States, Outputs and Disturbances
		}
	}

	\author{Shuangshuang Wu\thanks{S. Wu is with the Institute of Electrical Engineering, Yanshan University, Qinhuangdao, 066004, China. e-mail: {\tt \small ssw\underline{\hbox to 2mm{}}0538\underline{\hbox to 2mm{}}ysu@163.com } }, Sachin Shivakumar, Matthew~M.~Peet%
		\thanks{S. Shivakumar, M. Peet are with the School for the Engineering of Matter, Transport and Energy, Arizona State University, Tempe, AZ, 85298 USA. e-mail: {\tt \small sshivak8@asu.edu, \tt \small mpeet@asu.edu } }, Changchun Hua\thanks{C. Hua is with the Institute of Electrical Engineering, Yanshan University, Qinhuangdao, 066004, China. e-mail: {\tt \small cch@ysu.edu.cn } }
	}

	\maketitle

\begin{abstract}              
	This paper considers the $H_\infty$-optimal estimation problem for linear systems with multiple delays in states, output, and disturbances. First, we formulate the $H_\infty$-optimal estimation problem in the Delay-Differential Equation (DDE) framework. Next, we construct an equivalent Partial Integral Equation (PIE) representation of the optimal estimator design framework. We then show that in the PIE framework, the $H_\infty$-optimal estimator synthesis problem can be posed as a Linear PI Inequality (LPI). LPIs are a generalization of LMIs to the algebra of Partial Integral (PI) operators and can be solved using the PIETOOLS toolbox. Finally, we convert the PIE representation of the optimal estimator back into an ODE-PDE representation - a form similar to a DDE, but with corrections to estimates of the infinite-dimensional state (the time-history). Numerical examples show that the synthesis condition we propose produces an estimator with provable $H_\infty$-gain bound which is accurate to 4 decimal places when compared with results obtained using Pad\'e-based discretization.
\end{abstract}

\section{Introduction}

In most practical control scenarios, feedback control requires the use of sensors to measure the current state of the system. However, such sensors are often noisy and can measure only a small subset of the required state variables~\cite{Smith:57,Zhou:13,Wang:19}. For Ordinary Differential Equations (ODEs), the problem of state estimation has been largely solved, with special cases including the Luenberger observer~\cite{Luenberger:71}, the Kalman Filter~\cite{Kalman:95}, and the LMI for $H_\infty$-optimal state estimation~\cite{Fridman:01}. When delays are introduced, however (e.g. in state, input or output), estimators designed for the undelayed ODE can destabilize if applied to the resulting Delay-Differential Equation (DDE) System. Consequently, the problem of designing stable or optimal observers for systems with delay has received significant attention in recent years - See, e.g. \cite{Zhou:13,Fridman:01} and the references therein.  Specifically, suppose we consider the problem of designing a state estimator for a DDE system of the form

\vspace{-0.5cm}
{\small{
\begin{align}
&\dot{x}(t)=A_0x(t)+Bw(t)+\sum_{i=1}^K(A_{i}x(t-\tau_i)+B_iw(t-\tau_i))\notag \\
&z(t)=C_{1}x(t)+D_1w(t)+\sum_{i=1}^K(C_{1i}x(t-\tau_i)+D_{1i}w(t-\tau_i))\notag\\
&y(t)=C_{2}x(t)+D_2w(t)+\sum_{i=1}^K(C_{2i}x(t-\tau_i)+D_{2i}w(t-\tau_i))
\label{eqn:nominal}
\end{align}}} where $x(t)\in \R^n$ is the state, $w(t)\in \R^r$ is an external disturbance input with $w(0)=0$,  $z(t)\in \R^{p}$ is the regulated output, and $y(t)\in \R^{q}$ is the measured output.

Most work on estimator design has assumed the estimator itself is modelled as a DDE - e.g.~\cite{Fridman:01,Fattouh:00}. However, an alternative body of work has recently emerged which argues that the optimal estimator for a DDE is not itself a DDE, but rather an ODE coupled with PDE~\cite{Peet:19_estimator}, where the PDE represents not only transport, but also allows for corrections to the estimate of the state-history of the DDE (See $\hat{\mbf x} $ in Eqn.~\eqref{Eqn:x_hatx_1}). In this context, SOS and backstepping methods for observer design in the coupled ODE-PDE framework have been developed wherein the observer simultaneously estimates both the current state and the history of the state~\cite{Krstic:08,Ali:18,Peet:19_estimator}.

Unfortunately, however, there several limitations in these previous efforts. Specifically, while the backstepping transformation method applied in \cite{Krstic:08} is guaranteed to produce a stable estimator if one exists, this estimator is not guaranteed to be optimal in any sense. Meanwhile, the SOS methods employed in~\cite{Peet:19_estimator} and~\cite{Peet:14}, while highly accurate and similar in structure to the observers proposed in this paper, were unable to handle delays in the inputs or outputs. It is our belief that this inability to handle input and output delays is due to problems in formulation arising from use of the coupled ODE-PDE representation. Specifically, disturbances and output which occur at the boundary of the transport equations are not well-modeled when using either as inputs/outputs to either ODE or PDE state (current state or history). For this reason, in this manuscript, we do not use the ODE-PDE representation, but rather look to the Partial Integral Equation (PIE) representation of the DDE system, wherein boundary conditions are not auxiliary, but rather eliminated by incorporating their effect directly into the dynamics of the system.

For an example illustrating the importance of signal delays in estimator design, we refer to the model of a mining cable elevator system in~\cite{Wang:19}, which has sensor output delays due to wireless propagation delay from the elevator at the bottom of the shaft (over 2000 m underground) to the control center on the ground, and disturbance delays, where the disturbances are vibrations caused by deformation of the cage and is coupled to the shaft structure and cable tension. Failure to accurately account for output and disturbance delays can lead to chattering effects or even instability.

Having motivated the need for optimal observer design for systems with input, output, and state-delay, we now turn to the contributions of this manuscript. Our first contribution is to reformulate the nominal DDE in Eq.~\eqref{eqn:nominal} as an equivalent Partial Integral Equation (PIE)~\cite{Peet:19_tds}. By eliminating boundary conditions, the PIE representation allows us to model the effect of disturbances on the dynamics and consequently pose an $H_\infty$-optimal estimator design problem wherein the objective is to design an $H_{\infty}$-optimal estimator which uses the measured output $y$ to construct an estimate of $x$ and $z$ while minimizing $\gamma:=\text{sup}_{w\in L_2}\frac{\lVert z_e\rVert_{L_2}}{\lVert w\rVert_{L_2}}$, where $z_e(t)=\hat z(t)-z(t)$ denotes the error between $z(t)$ and its estimate $\hat z(t)$. Furthermore, the use of the PIE formulation  - parameterized by Partial Integral (PI) operators allows us to generalize the LMI for $H_\infty$-optimal estimation of ODEs to a convex Linear PI Inequality (LPI) which solves the $H_\infty$-optimal observer synthesis problem for the given class of PIEs. Next, we solve the resulting LPI for optimal observer synthesis using the PIETOOLS Matlab toolbox~\cite{Sachin:19}. Finally, we take the resulting estimator, formulated as a PIE, and convert this back to a coupled ODE-PDE in order to allow for efficient implementation.

\vspace{-0.1cm}
\subsection{Notation}
Shorthand notation used throughout this paper includes the Hilbert spaces $L_2^m[X]:=L_2(X;\R^m)$ of square integrable functions from $X$ to $\R^m$ and $W_2^m[X]:=W^{1,2}(X;\R^m)=H^1(X;\R^m)=\{x:x,\partial_s x\in L_2^m[X]\}$; $L_2^m$ and $W_2^m$ are used when domains are clear from context. Furthermore, the extension $W_2^{n\times m}[X]:=W^{1,2}(X;\R^{n\times m})$ is used to denote matrix-valued functions. $I$ denotes the identity matrix. A block-diagonal matrix is denoted by $\text{diag}\{\cdots\}$. An operator $\mathcal{P}:Z\rightarrow Z$ is positive on a subset $X$ of Hilbert space $Z$ if $\left< x, \mathcal{P}x\right>_Z\geq 0$ for all $x\in X$. $\mathcal{P}$ is coercive on $X$ if $\left< x, \mathcal{P}x\right>_Z\geq\epsilon\lVert x \rVert^2_Z$ for some $\epsilon>0$ for all $x\in X$. If $\mathcal{P}^{1}$ and $\mathcal{P}^{2}$ are two linear operators then $\left(\mathcal{P}^{1}\right)^{*}$ stands
for the adjoint of $\mathcal{P}^{1}$ and $\mathcal{P}^{1} \mathcal{P}^{2}$ represents composition of those operators in shown order. For brevity, symmetric components of a block-operator are denoted by $(\cdot)$ and adjoints by $(\cdot)^*$. The space $Z_{m,n}$:=$\R^m\times$ $L_2^{n}[-1,0]$ is an inner-product space with the inner product defined as
\begin{align*}
\left<
\bmat{
	y\\
	\psi
},\bmat{
	x\\
	\phi
}
\right>_{Z_{m,n}}=y^Tx+\int_{-1}^0\psi(s)^T\phi(s)ds,
\end{align*} where $x, y \in \R^m$ and $\psi, \phi \in L_2^{n}[-1,0]$.

\section{Linear PIE Representation Of Time-Delay Systems}
In this section, we present the PIE representation of time delay systems. PIE representation is used, instead of ODE-PDE representation, because the operators in PIE representation are bounded. Furthermore, unlike coupled ODE-PDE representation of time-delay systems, PIE representation do not require boundary conditions and the solution of PIE systems do not have additional continuity constraints.
\subsection{Linear PIEs}
A general class of linear PIEs system is defined as follows
\begin{align}
\mcl T\dot{\mbf{x}}(t)&+\mcl B_T\dot w(t)=\mcl{A}\mbf{x}(t)+\mcl{B} \omega (t)\notag\\
z(t)&=\mcl{C}_1\mbf{x}(t)+\mcl D_1 \omega(t)\notag\\
y(t)&=\mcl{C}_2\mbf{x}(t)+\mcl D_2 \omega(t) \label{eqn:DPS}
\end{align}
where $\mcl{T}$, $\mcl{B_T}$, $\mcl{A}$, $\mcl{C}_1$, $\mcl{C}_2$, $D_1$ and $\mcl{D}_2$ are all \textbf{Partial Integral (PI) operators} with the following form
\begin{align}
&\left(\fourpi{P}{Q_1}{Q_2}{R_i}\bmat{
	x  \\
	\phi
}\right)(s):=\label{eqn:PQRS}\notag\\&\hspace{0.8cm}\bmat{
	Px+\int_{-1}^0Q_{1}(s)\phi(s)ds\\\quad \\(Q^T_{2}(s)x+R_{0}(s)\phi(s)+\int_{-1}^sR_{1}(s,\theta)\phi(\theta)d\theta\\+\int_{s}^0R_{2}(s,\theta)\phi(\theta)d\theta)
}.
\end{align}

For any given $w\in W^{1,2}[0,\infty)$, and
$\mbf x_0\in \R\times L_2[-1,0]$, suppose
$\mbf x$ is  Fr\'echet differentiable almost everywhere on $[0,\infty)$, $\mbf x(0)=\mbf x_0$, and Eq. \eqref{eqn:DPS} are satisfied for all $t\geq 0$. Then $\mbf x(t):[0, \infty)\rightarrow \R\times L_2[-1,0]$, $ z(t):[0, \infty)\rightarrow \R$, $ y(t):[0, \infty)\rightarrow \R$ satisfy the PIE defined by Eq.~\eqref{eqn:DPS_er}. For more details on PI operators, please see \cite{Sachin:19}.



\subsection{Representing Time Delay Systems as Linear PIEs}

Linear time delay systems in the representation of DDEs can be converted to linear PIEs for special definitions of the PI operators $\mcl T, \mcl B_T, \mcl A, \mcl B, \mcl C_1, \mcl C_2, \mcl D_1, \mcl D_2$. For DDEs~\eqref{eqn:nominal} defined by $\tau_i$ and the matrices $A_i$, $B_i$, $C_i$, $C_{ij}$, $D_i$ and $D_{ij}$  we define the PI operators
\begin{align}
\mcl T&:=\fourpi{I}{0}{T_0}{0,0,-I},& \mcl B_T&:=\fourpi{0}{0\mkern-9.5mu/}{T_1}{0\mkern-9.5mu/}, \notag\\
\mcl A&:=\fourpi{A_0+\sum_{i=1}^K A_i }{\tilde A}{0}{H,0,0},
& \mcl B&:=\fourpi{B+\sum_{i=1}^KB_{i}}{0\mkern-9.5mu/}{0}{0\mkern-9.5mu/}, \notag
\\
\mcl C_1&:=\fourpi{C_{1}+\sum_{i=1}^KC_{1i}}{\tilde C_1}{0\mkern-9.5mu/}{0\mkern-9.5mu/},
& \mcl C_2&:=\fourpi{C_{2}+\sum_{i=1}^KC_{2i}}{\tilde C_2}{0\mkern-9.5mu/}{0\mkern-9.5mu/}, \notag
\\\mcl D_1&:=\fourpi{D_1+\sum_{i=1}^KD_{1i}}{0\mkern-9.5mu/}{0\mkern-9.5mu/}{0\mkern-9.5mu/},
& \mcl D_2&:=\fourpi{D_2+\sum_{i=1}^KD_{2i}}{0\mkern-9.5mu/}{0\mkern-9.5mu/}{0\mkern-9.5mu/} \label{Eqn:PIs}.
\end{align}
where
{\small{
\begin{align}
&C_{ri}=\bmat{I\\0}, \quad B_{ri}=\bmat{0\\I},  \quad A_{ki}=\bmat{A_i&B_i}, \notag\\
&C_{k1i}=\bmat{C_{1i}&D_{1i}},  \quad C_{k2i}=\bmat{C_{2i}&D_{2i}},\notag\\
&T_0=\bmat{C_{r1}\\ \vdots\\ C_{rK}}, \quad
T_1=\bmat{B_{r1}\\ \vdots\\ B_{rK}}, \quad H=\text{diag}\left\{\frac{1}{\tau_1}I,\cdots,\frac{1}{\tau_K}I \right\}, \notag\\
&\tilde A=-\bmat{A_{k1},&\cdots,&A_{kK}},\quad
\tilde C_1=-\bmat{C_{k11},&\cdots,&C_{k1K}}, \notag\\
& \tilde C_2=-\bmat{C_{k21},&\cdots,&C_{k2K}}.\label{Eqn:ACC}
\end{align}}}
We give the following lemma.

\begin{lem}
	\label{lem:ODEPDE_PIE} Suppose $\mcl T, \mcl B_T, \mcl A, \mcl B, \mcl C_1, \mcl C_2, \mcl D_1, \mcl D_2$ are as defined in Eq.~\eqref{Eqn:PIs}. For given $w\in W^{1,2}[0,\infty)^r$, if $x$, $z$, and $y$ satisfy Eq.~\eqref{eqn:nominal}, then $z$ and $y$ also satisfy the PIE \eqref{eqn:DPS} with
	\begin{align}
	\mbf x(t)=\bmat{x(t)\\\partial_s\phi_1(t,\cdot)\\\vdots\\\partial_s\phi_K(t,\cdot)}\label{Eqn:x}
	\end{align} where $\phi_i(t,s)=C_{ri}x(t+\tau_is)+B_{ri}w(t+\tau_is)$ for $C_{ri}$ and $B_{ri}$ as defined in \eqref{Eqn:ACC}.
	Furthermore, if $\mbf x$, $z$ and $y$ satisfy the PIE defined by Eq.\eqref{eqn:DPS},
	then $x, z$ and $y$ satisfy Eq.~\eqref{eqn:nominal} where
	\begin{align}
	\bmat{x(t)\\\cdot}=\mcl T\mbf x(t)+\mcl B_Tw(t).
	\end{align}
\end{lem}
\vspace{0.2cm}
\begin{proof}
	For given $w\in W^{1,2}[0,\infty)^r$, suppose $x$, $z$, and $y$ satisfy the DDEs defined by Eq.~\eqref{eqn:nominal}. Define	$\phi_i(t,s)=C_{ri}x(t+\tau_is)+B_{ri}w(t+\tau_is)$ where $C_{ri}, B_{ri}$ are as defined in Eq.~\eqref{Eqn:ACC}.Then from Lemma 3 in \cite{Peet:19_tds}, we get
	$x$, $z$, and $y$ satisfy the following ODE-PDE Eq. \eqref{eqn:PDE_ODE} and vice versa.
	\begin{align}
	&\dot x(t)=A_0x(t)+Bw(t)+B_vv(t)\notag\\
	&z(t)=C_{1}x(t)+D_1w(t)+D_{1v}v(t)\notag\\
	&y(t)=C_{2}x(t)+D_2w(t)+D_{2v}v(t)\notag\\
	&\dot \phi_i(t,s)=\frac{1}{\tau_i}\partial_s \phi_i(t,s),\notag\\
	&\phi(t,0)=C_{ri}x(t)+B_{ri}w(t)\notag\\
	&v(t)=\sum_{i=1}^K C_{vi}\phi_i(t,-1)
	\label{eqn:PDE_ODE}
	\end{align}
	where
	\begin{align}
	&B_v=\bmat{I&0&0},  \ D_{1v}=\bmat{0&I&0},\notag\\
	&D_{2v}=\bmat{0&0&I},  \  C_{vi}=\bmat{A_i&B_i\\C_{1i}&D_{1i}\\C_{2i}&D_{2i}}.	\label{eqn:bv}
	\end{align} 	
	
	Suppose $\mbf x(t)\in Z_{n,K(n+r)}$ is defined as Eq.~\eqref{Eqn:x}, where $x$, $\phi_i$ satisfy the ODE-PDE form~\eqref{eqn:PDE_ODE}, and the PI operators are as defined in Eq.~\eqref{Eqn:PIs}, Then,
	from Lemma 4 in \cite{Peet:19_tds}, we get that $\mbf x$, $z$ and $y$ also satisfy PIEs~\eqref{eqn:DPS} and vice versa. This completes the proof.
\end{proof}

\section{Estimation of Linear PIEs}
For the linear PIEs Eq. \eqref{eqn:DPS}, an estimator in the PIE form is constructed. The coupled system dynamics are as follows
\begin{align}
\mcl T\dot{\mbf{x}}(t)&+\mcl B_T\dot w(t)=\mcl{A}\mbf{x}(t)+\mcl{B} \omega (t),\notag\\
z(t)&=\mcl{C}_1\mbf{x}(t)+\mcl D_1 \omega(t)\notag\\
y(t)&=\mcl{C}_2\mbf{x}(t)+\mcl D_2 \omega(t),
\notag\\
\mcl T\dot{\hat{\mbf{x}}}(t)&=\mcl{A}\hat{\mbf{x}}(t)+\mcl{L}(\hat{y}(t)-y(t))\notag\\
\hat z(t)&=\mcl{C}_1\hat{\mbf{x}}(t)\quad
\hat{y}(t)=\mcl{C}_2\hat{\mbf{x}}(t),\notag\\
\mbf{x}(0)&=\hat{\mbf{x}}(0)=\mbf x_0\in Z \label{eqn:couple_1}
\end{align}
where $\mcl{L}:\R\rightarrow Z$ is a PI operator. Let $\mbf{e}(t):=\hat{\mbf{x}}(t)-\mbf{x}(t)$, then we have
\begin{align}
\mcl T\mbf{\dot{e}}(t)-\mcl B_T\dot w(t)&=(\mcl{A}+\mcl{L}\mcl{C}_2)\mbf{e}(t)-(\mcl{B}+\mcl{L}\mcl D_2) \omega(t)\notag\\
z_e(t)&=\mcl{C}_1\mbf{e}(t)-\mcl D_1\omega(t).\label{eqn:DPS_er}
\end{align}
Then $\mbf e(0)=0$ and the LMI conditions in KYP Lemma for linear ODEs can be extended to linear PIEs using the LPI conditions.
\begin{thm} \label{thm:DPS}
	Suppose there exists a scalar $\gamma>0$ and bounded linear operators $\mcl{P}:Z\rightarrow Z$ is bounded, self-adjoint, coercive and $\mcl{Z}: \R\rightarrow Z$ such that
	\begin{align}
	&\bmat{\mcl B_T^*(\mcl P\mcl B+\mcl Z\mcl D_2)+(\cdot)^*& 0 &(\cdot)^*\\ 0  &0 & 0 \\ -(\mcl P\mcl A+\mcl Z\mcl C_2)^*\mcl B_T& 0 &0}\notag\\
	&\hspace{0.4cm}+\bmat{-\gamma I& -\mcl D_1^T&-(\mcl P\mcl B+\mcl Z\mcl D_2)^*\mcl T\\(\cdot)^*&-\gamma I&\mcl C_1\\(\cdot)^*&(\cdot)^*&(\mcl P\mcl A+\mcl Z\mcl C_2)^*\mcl T+(\cdot)^*}<0.
	\label{eqn:thm_dps}
	\end{align}		
	Then $\mcl{P}^{-1}$ exists and is a bounded linear operator. For any given $w\in W_2$,
	if $z$ and $\hat z$ satisfy Eqn. \eqref{eqn:couple_1} where $\mcl{L}=\mcl{P}^{-1}\mcl{Z}$ for some $\mbf x$ and $\hat{\mbf x}$, then we have
	$\lVert z_e\rVert_{L_2}\leq\gamma\lVert\omega\rVert_{L_2}$ where $z_e(t) = \hat{z}(t)-z(t)$.
\end{thm}

\begin{proof}
Suppose there exist $\gamma$, $\mcl{P}$ and $\mcl{Z}$ that satisfy the assumptions of the Theorem statement and let $\mcl{L}= \mcl{P}^{-1}\mcl{Z}$.
Define the storage functional
\[V(t)=\ip{\mcl T\mbf e(t)-\mcl B_T w(t)}{\mcl P (\mcl T\mbf e(t)-\mcl B_T w(t))}_Z.\]

Obviously,
$V(t)\ge \delta \norm{\mcl{T}\mbf e(t)-\mcl{B}_Tw(t)}_Z^2$
holds for some $\delta>0$ since $\mcl P$ is coercive. Since $\mcl P$ is bounded, self-adjoint, coercive, from Theorem 1 in \cite{Peet:19_control},  $\mcl{P}^{-1}$ exists and is a bounded linear operator. Then, for $\mbf{e}$ and $z_e$ that satisfy \eqref{eqn:DPS_er},
\begin{align*}
&\dot V(t)-\gamma\lVert\omega(t)\rVert^2-\gamma\lVert\upsilon_e(t)\rVert^2+2\ip{\upsilon_e(t)}{ z_e(t)}_Z\\
&=\ip{\mcl T\mbf{e}(t)-\mcl B_T w(t)}{(\mcl P\mcl A+\mcl {ZC}_2)\mbf e(t)}_Z\\&\quad+\ip{(\mcl P\mcl A+\mcl {ZC}_2)\mbf e(t)}{\mcl T\mbf{e}(t)-\mcl B_T w(t)}_Z\\&\quad-\ip{\mcl T\mbf{e}(t)-\mcl B_T w(t)}{(\mcl P\mcl{B}+\mcl Z\mcl D_2) w(t)}_Z\\
&\quad-\ip{\mcl T\mbf{e}(t)-\mcl B_T w(t)}{(\mcl P\mcl{B}+\mcl Z\mcl D_2) w(t)}_Z\\
&\quad-\ip{(\mcl P\mcl{B}+\mcl Z\mcl D_2) w(t)}{\mcl T\mbf{e}(t)-\mcl B_T w(t)}_Z\\
&\quad-\gamma\lVert\omega(t)\rVert^2-\gamma\lVert\upsilon_e(t)\rVert^2+\left<\upsilon_e(t),\mcl{C}_1\mbf e(t)\right>\\&\quad+\left<\mcl{C}_1\mbf e(t),\upsilon_e(t)\right>-\left<\upsilon_e(t),\mcl D_1\omega(t)\right>-\left<\mcl D_1\omega(t),\upsilon_e(t)\right>\\
&=\bmat{w(t)\\v_e(t)\\\mbf x(t)}^T\Bigg(\bmat{\mcl B_T^*(\mcl P\mcl B+\mcl Z\mcl D_2)+(\cdot)^*& 0 &(\cdot)^*\\ 0  &0 & 0 \\ -(\mcl P\mcl A+\mcl Z\mcl C_2)^*\mcl B_T& 0 &0}\\
&\hspace{0.4cm}+\bmat{-\gamma I& -\mcl D_1^T&-(\mcl P\mcl B+\mcl Z\mcl D_2)^*\mcl T\\(\cdot)^*&-\gamma I&\mcl C_1\\(\cdot)^*&(\cdot)^*&(\mcl P\mcl A+\mcl Z\mcl C_2)^*\mcl T+(\cdot)^*}\Bigg )\bmat{w(t)\\v_e(t)\\\mbf x(t)}.
\end{align*} Set $\upsilon_e(t)=\frac{1}{\gamma} z_e(t)$. If Eq.~\eqref{eqn:thm_dps} is satisfied,
then
\begin{align*}
&\dot V(t)-\gamma\lVert\omega(t)\rVert^2+\frac{1}{\gamma}\lVert z_e(t)\rVert^2\leq 0.
\end{align*}
Integration of this inequality with respect to $t$ yields
\[
V(t)-V(0)-\gamma \int_0^t\norm{ w(s)}^2ds+\frac{1}{\gamma}\int_0^t
\norm{z_e(s)}^2 ds\leq 0.
\]

$V(0)=0$ and $V(t)\ge 0$ for any $t\ge0$. Then as $t\to\infty$, we get $\lVert  z_e\rVert_{L_2}\leq\gamma\lVert \omega\rVert_{L_2}$.
%
\end{proof}	
\vspace{-0.1cm}
\section{Estimation of time delay systems}
In this section, the estimator is constructed and using Theorem \ref{thm:DPS}, we get the $H_\infty $ estimation condition of time delay systems defined by Eq.~\eqref{eqn:nominal}.


\subsection{Coupling the DDEs and Estimator Dynamics }
For the plant system \eqref{eqn:nominal} restated here, we construct the estimator dynamics as a ODE-PDE coupled system. The coupled system dynamics are as follows,

\vspace{-0.5cm}
{\small{
\begin{align}
&\dot{x}(t)=A_0x(t)+Bw(t)+\sum_{i=1}^K(A_{i}x(t-\tau_i)+B_iw(t-\tau_i))\notag \\
&z(t)=C_{1}x(t)+D_1w(t)+\sum_{i=1}^K(C_{1i}x(t-\tau_i)+D_{1i}w(t-\tau_i))\notag\\
&y(t)=C_{2}x(t)+D_2w(t)+\sum_{i=1}^K(C_{2i}x(t-\tau_i)+D_{2i}w(t-\tau_i))\notag\\
&\dot {\hat x}(t)=A_0\hat x(t)+B_v\hat v(t)+L_1(\hat y(t)-y(t))\notag\\
&\hat z(t)=C_{1}\hat x(t)+D_{1v}\hat v(t)\notag\\
&\hat y(t)=C_{2}\hat x(t)+D_{2v}\hat v(t)\notag\\
&\dot {\hat \phi}_i(t,s)=\frac{1}{\tau_i}\partial_s \hat \phi_i(t,s)+L_{2i}(s)(\hat y(t)-y(t))\notag\\
&\hat \phi(t,0)=C_{ri}\hat x(t), \notag\\
&\hat v(t)=\sum_{i=1}^K C_{vi}\hat \phi_i(t,-1)
\label{eqn:couple_2}
\end{align}}}
where $\hat{x}(t)$, $\hat{z}(t)$ and $\hat{y}(t)$ are the estimates of $x(t)$, $z(t)$ and $y(t)$, respectively. The matrix $L_1$ and the polynomials $L_{2i}$ are observer gains to be determined.
The matrices $B_v, D_{iv}, C_{ri}, C_{vi}$ are the same ones used to define the ODE-PDE model~\eqref{eqn:PDE_ODE}.

The structure of the estimator allows us to represent Eq. \eqref{eqn:couple_2} as coupled linear PIE \eqref{eqn:couple_1} defined by the PI operators in Eq.~\eqref{Eqn:PIs}, where $\mcl L= \fourpi{L_1}{0\mkern-9.5mu/}{\bmat{L_{21}\\\vdots\\L_{2K}}}{0\mkern-9.5mu/}$. The equivalence between Eq.~\eqref{eqn:couple_2} and Eq.~\eqref{eqn:couple_1} is stated as the Lemma~\ref{lem:Equivalence} in Appendix.


\subsection{Applying Theorem~\ref{thm:DPS} to time delay systems}

\begin{thm}\label{thm:Thm_f}
	Suppose there exists positive scalar $\gamma$, matrix $P\in \R^{n\times n}$, matrix $H$, $\Gamma$, $W$ with appropriate dimensions,  polynomial $Z(s)$, function $R_0 \in W_2^{ns\times ns}[-1,0]$ with $ns=K(n+r)$,
	matrix $Z_1\in \R^{n\times q}$, such that	
	the operator $\mcl P:=\fourpi{P}{HZ(s)}{Z(s)^TH^T}{ R_i}$ with $R_2=R_1=Z(s)^T\Gamma Z(\theta)$ is bounded, self-adjoint, and coercive, and $\mcl Z:=\fourpi{Z_1}{0\mkern-9.5mu/}{Z(s)^TW}{0\mkern-9.5mu/}$ satisfy
	\begin{align}
	&\bmat{\mcl B_T^*(\mcl P\mcl B+\mcl Z\mcl D_2)+(\cdot)^*& 0 &(\cdot)^*\\ 0  &0 & 0 \\ -(\mcl P\mcl A+\mcl Z\mcl C_2)^*\mcl B_T& 0 &0}\notag\\
	&\hspace{0.4cm}+\bmat{-\gamma I& -\mcl D_1^T&-(\mcl P\mcl B+\mcl Z\mcl D_2)^*\mcl T\\(\cdot)^*&-\gamma I&\mcl C_1\\(\cdot)^*&(\cdot)^*&(\mcl P\mcl A+\mcl Z\mcl C_2)^*\mcl T+(\cdot)^*}<0.
	\label{eqn:thm_dps_2}
	\end{align}		
	where the operators $\mcl T, \mcl B_T, \mcl A, \mcl B, \mcl C_1,\mcl C_2, \mcl D_1, \mcl D_2$ are defined as Eqn. \eqref{Eqn:PIs}.
	Then for any given $w\in W^{1,2}[0,\infty)^r$,
	if $z(t)$ and $\hat z(t)$
	satisfy the Eq.~\eqref{eqn:couple_2} where
	
	\vspace{-0.5cm}
	{\small{
			\begin{align}
			L_{1}&=\left(I- \hat H K H^T\right)P^{-1} Z_{1}+\hat{H} K W\notag\\
			\bmat{L_{21}\\\vdots\\L_{2K}}(s)&=R_0(s)^{-1} Z(s)^{T}\left(\hat{H}^{T} Z_{1}+W+\hat{\Gamma}KW\right)\label{Eqn:L_solved}
			\end{align}
			and
			\begin{align*}
			\hat H &=P^{-1} H \left( K H^T P^{-1} H -I- K\Gamma\right)^{-1}\\ K&=\int_{-1}^{0} Z(s)R_0(s)^{-1} Z(s)^{T} ds,\\
			\hat \Gamma&=-(\hat H^T H+\Gamma)(I + K \Gamma)^{-1},
			\end{align*}}} for some $x$, $\hat x$ and $\hat \phi_i$,
	define $z_e(t) = \hat{z}(t)-z(t)$, then we have
	$\lVert z_e\rVert_{L_2}\leq\gamma\lVert \omega\rVert_{L_2}$.
\end{thm}

\begin{proof}
	Suppose there exists $\gamma$, matrices $P$, $Z_1$, $H$, $\Gamma$ and $W$, polynomial $Z$ and function $R_0$ such that $\mcl P$, as defined in the Theorem statement, is bounded and coercive and satisfies the LPI \eqref{eqn:thm_dps_2}. Further, for given $w\in W^{1,2}[0,\infty)^r$, let $z$ and $\hat z$
	satisfy the Eq.~\eqref{eqn:couple_2}, where $L_1$ and $L_{2i}$ are as defined in Eq.~\eqref{Eqn:L_solved}, for some $x$ and $\hat x$.
	
	Then $\mcl P^{-1}$ exists, is bounded and using Lemma \ref{lem:inverse} in Appendix, $\mcl{P}^{-1}$ is
	\begin{align}
	{\mcl P}^{-1}:=\fourpi{\hat P}{\hat Q}{\hat Q^T}{\hat R_i}
	\end{align} where
	\begin{align*}
	\hat P & = \left(I- \hat H K H^T\right)P^{-1},\qquad    \hat Q(s)  =\hat H Z(s)R_0(s)^{-1} \\
	\hat R_0(s) & = R_0(s)^{-1},\quad   \hat R_{1}(s,\theta)=\hat R_0^T(s) Z(s)^T\hat \Gamma Z(\theta)\hat R_0(\theta).
	\end{align*}
	Define the PI operator $\mcl{L}$ as
	\begin{align}
	\mcl L= \fourpi{L_1}{0\mkern-9.5mu/}{\bmat{L_{21}\\\vdots\\L_{2K}}}{0\mkern-9.5mu/}
	\end{align}
	where $L_1$ and $L_{2i}$ are as defined in Eq.~\eqref{Eqn:L_solved}.
	Then, from Lemma \ref{lem:gains}, $\mcl{L}=\mcl{P}^{-1}\mcl{Z}$. Thus, $\mcl L$, $\mcl{P}$ and $\mcl{Z}$, satisfy the conditions of Theorem \ref{thm:DPS}.
	
	Since $z$ and $\hat z$ satisfy Eq.~\eqref{eqn:couple_2} for some $x$, $\hat x$ and $\hat \phi_i$, from Lemma~\ref{lem:Equivalence}, we get $z(t)$ and $\hat z(t)$ also satisfy the Eq.~\eqref{eqn:couple_1} for
	\begin{align}
	\mbf x(t)=\bmat{x(t)\\\partial_s\phi_1(t,\cdot)\\\vdots\\\partial_s\phi_K(t,\cdot)}, \quad
	\hat{\mbf x}(t)=\bmat{\hat x(t)\\\partial_s\hat \phi_1(t,\cdot)\\\vdots\\\partial_s\hat \phi_K(t,\cdot)}. \label{Eqn:x_hatx_1}
	\end{align}
	where $\phi_i(t,s)=C_{ri}x(t+\tau_is)+B_{ri}w(t+\tau_is)$.
	We conclude that $z$ and $\hat{z}$ satisfy the conditions of Theorem~\ref{thm:DPS} for the operators $\mcl{P}$, $\mcl{Z}$ and $\mcl{L}$ as defined. Since all conditions of Theorem~\ref{thm:DPS} are satisfied, we conclude that
	$\lVert z_e\rVert_{L_2}\leq\gamma\lVert\omega\rVert_{L_2}$ where $z_e(t) = \hat{z}(t)-z(t)$.
\end{proof}	

\section{Numerical Implementation and Examples}
The LPI in Theorem~\ref{thm:Thm_f} is implemented using the Matlab PIETOOLS toolbox, wherein we minimize $\gamma$, the closed-loop $H_\infty$-performance gain. This toolbox is available online for validation or download from Code Ocean \cite{Sachin:19}. PIETOOLS allows for declaration of PI variables, PI inequality constraints, and manipulation of PI operators as an object class. A selection of the code from this implementation is as follows.

\vspace{-0.5cm}
{\small{
		\begin{align*}
		&\texttt{>> pvar s th gam;}\\
		&\texttt{>> opvar T Bt A B C1 C2 D1 D2;}\\
		&\texttt{>> S=sosprogram([s,th],gam);}\\
		&\texttt{>> [H,P] = sos\_posopvar(H,dim1,X,s,th);}\\
		&\texttt{>> [H,Z] = sos\_opvar(H,dim2,X,s,th);}\\
		&\texttt{>> F1=P*B+Z*D2; F2=P*A+Z*C2;}\\
		&\texttt{>> E = -gam*eye(r)-Bt'*F1-F1'*Bt};\\
		&\texttt{>> Df =[\hspace{0.1cm} E \hspace{1.1cm}  -D1' \ -F1'*T-Bt'*F2;}\\
	    &\texttt{ \hspace{1.3cm}  -D1 \hspace{0.6cm} -gam*eye(p) \quad     C1;}\\
	    &\texttt{ \hspace{0.6cm} -T'*F1-F2'*Bt \ C1' \hspace{0.2cm} F2'*T+T'*F2];}\\
		&\texttt{>> H = sosopineq(H,-Df);}\\
		&\texttt{>> H = sossetobj(H,gam);}\\
		&\texttt{>> H = sossolve(H);}
		\end{align*}
}}
For simulation, a fixed-step forward-difference-based discretization
method is used, with a different set of states
representing each delay channel. In the simulation results
given below, 100 spatial discretization points are used for each
delay channel.

We have applied the resulting code to several representative examples. In each case, we list: $\gamma_{min}$ - the provable bound on the $L_2$-gain from the disturbance $w$ to the regulated output $z_e$ of the $H_\infty$-optimized observer obtained from the LPI; $\gamma_{pade}$ - an estimated achievable $L_2$-gain obtained using LMI methods and a 10th Pad\'e ODE approximation of the DDE; and $\gamma_{real}$ - the observed $L_2$-gain bound obtained by applying a simulation of our optimized estimator to a simulation of the nominal DDE with a representative disturbance signal. Note that because there are no works which address the problem of $H_\infty$-optimal control of systems with input, output, and state delay, we are not able to compare our results with existing literature. However, this is not because of sub-optimality, and indeed, our estimators match or significantly outperform all other estimators when applied to systems lacking input or output delay.

\paragraph{Example 1} The following system is a variation of an example in~\cite{Peet:19_estimator},
\begin{align*}
\dot x(t)&=\bmat{0 &0\\0& 1}x(t)+\bmat{-1 &-1\\0&0.9}x(t-1)+\bmat{1&0\\0&1}w(t)\\
z(t)&=\bmat{1 & 0}x(t)+\bmat{1&10}x(t-1)\\
y(t)&=\bmat{1&10}x(t-1)+\bmat{0&5}w(t-1)
\end{align*}

wherein we have added output and disturbance delay to the dynamics. In this case Theorem~\ref{thm:Thm_f} yields $\gamma_{min}=1.8081$. Meanwhile the Pad\'e approximation $\gamma_{pade}=1.8081$ - an exact match. Figure 1
displays the effect of a sinc disturbance $w(t)$ on error in states $e(t)=\hat x(t)-x(t)$ using our optimized estimator. For this step disturbance, the actual $L_2$-gain is found to be $\gamma_{real}=0.5876$ - consistent with the predicted worst-case performance bound.

\begin{figure}
	\centering
	\includegraphics[width=.4\textwidth]{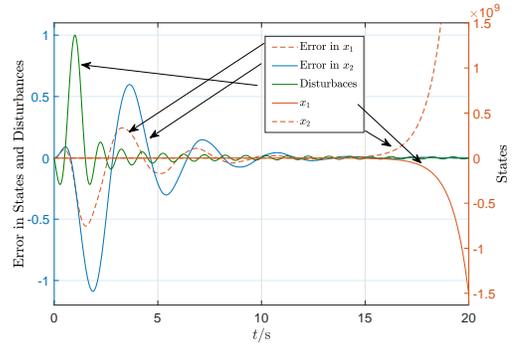}\vspace{-3mm}
	\caption{Response in error state to a sinc disturbance for E1}\label{fig:ex_b3}
\end{figure}

\paragraph{Example 2}
Consider now a slightly version of the Example in~\cite{de}.
\begin{align*}
&\dot x(t)=\bmat{0 &3\\-4& -5}x(t)+\bmat{-0.1 &0\\0.2&-0.2}x(t-0.3)\\
&\quad +\bmat{0&0.1\\-0.2&-0.3}x(t-0.5)+\bmat{-0.4545&0\\0&0.9090}w(t)\\
&y(t)=\bmat{0 & 100}x(t)+\bmat{0 &10}x(t-0.3)\\
&\quad \quad +\bmat{0 &2}x(t-0.5)+\bmat{1&1}w(t)\\
&z(t)=\bmat{0 & 100}x(t)
\end{align*}
wherein we have added an extra delay. In this case Theorem~\ref{thm:Thm_f} yields $\gamma_{min}=0.9592$. Meanwhile the Pad\'e approximation $\gamma_{pade}=0.9592$ - an exact match. Figure 2
displays the effect of a sinc disturbance $w(t)$ on error in states $e(t)=\hat x(t)-x(t)$ using our optimized estimator. For this step disturbance, the actual $L_2$-gain is found to be $\gamma_{real}=0.5792$ - consistent with the predicted worst-case performance bound.

\begin{figure}
	\centering
	\includegraphics[width=.35\textwidth]{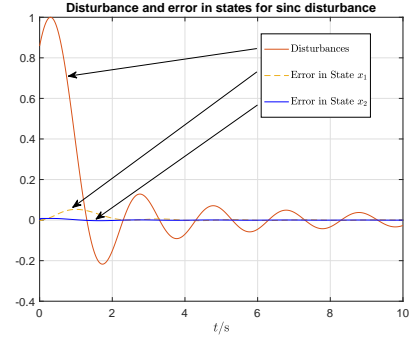}\vspace{-3mm}
	\caption{Response in error state to a sinc disturbance for E2}\label{fig:ex_b3}
\end{figure}

\paragraph{Example 3}
To test the scalability of our algorithm, we consider the following unstable n-D system with K delays, a single disturbance $w(t)$ and a single regulated $z(t)$ and a single sensed output $y(t)$.
\begin{align*}
\dot x(t)&=-\sum_{i=1}^K\frac{x(t-i/K)}{K}+\mbf 1w(t)\\
z(t)&=y(t)=\mbf 1^Tx(t)+\mbf 1^Tw(t)
\end{align*}
We examine how the computational complexity of the algorithm scales as the
product of the number of delays $K$ and number of states $n$ increases. Table I lists the computation time as CPU sec on a Intel i7-5960X processor omitting preprocessing and postprocessing times.

\begin{table}
	\caption{CPU Seconds of Sedumi Solving Process for $n$ States and $K$ Delays}
	\hbox{\begin{tabular}{c|c|c|c|c|c}
			\backslashbox{$K$}{$n$}  & $1$  & $2$ & $3$  & $4$ &6\\
			\hline
			$1$ &  0.3610  & 0.4630 & 8.488 &  1.887 & 16.50 \\\hline
			$2$ &  0.4380  & 1.573 & 11.94 &  77.94 & 950.8 \\
			\hline
			$3$ & 0.9000 & 10.14 & 167.0  & 913.9 &9827\\
			\hline
			$4$ & 1.331 & 82.92 & 912.6  & 4263&24030\\
			\hline
			$6$ & 12.10 & 967.2  & 9650  & 23980 &N/A\\
	\end{tabular}}
\end{table}


\section{Conclusion}
We have investigated the problem of $H_\infty$-optimal estimator design for systems with multiple delays in the states, outputs and disturbances. The commonly used DDE representation of nominal system and estimator is converted to a PIE representation. Within the PIE framework, we propose a convex formulation of the optimal estimator synthesis problem, in the form of an LPI - a form of convex optimization for which we have an efficient Matlab Toolbox. We then convert the optimized observer back into a coupled ODE-PDE for convenient implementation. Applying the results to several numerical examples, we find the resulting observers are non-conservative to 4 decimal places as measured against a Pad\'e-based ODE approximation of the DDE. Finally, the scalability of the algorithm is demonstrated for large numbers of delays and states.

\vspace{-0.15cm}
\begin{appendix}
	\subsection{Getting the inverse of $\mcl P$}
	\begin{lem}\label{lem:inverse}
		Suppose that $Q(s) = HZ(s)$ and $R_1(s,\theta)=Z(s)^T \Gamma Z(\theta)$ where $Z$ is a polynomial and $\mcl P:=\fourpi{P}{Q}{Q^T}{R_i}$ with $R_2=R_1$ is a coercive and self-adjoint operator where $\mcl P: X \rightarrow X$. Then ${\mcl P}^{-1}:=\fourpi{\hat P}{\hat Q}{\hat Q^T}{\hat R_i}$ with $\hat R_1=\hat R_2$
		where
		\begin{align*}
		\hat P & = \left(I- \hat H K H^T\right)P^{-1},\qquad    \hat Q(s)  =\hat H Z(s)R_0(s)^{-1} \\
		\hat R_0(s) & = R_0(s)^{-1},\quad   \hat R_{1}(s,\theta)=\hat R_0^T(s) Z(s)^T\hat \Gamma Z(\theta)\hat R_0(\theta),\\
		K&=\int_{-1}^0  Z(s)R_0(s)^{-1} Z(s)^Tds\\
		\hat H &=P^{-1} H \left( K H^T P^{-1} H -I- K\Gamma\right)^{-1}\\
		\hat \Gamma&=-(\hat H^T H+\Gamma)(I + K \Gamma)^{-1}.
		\end{align*}
		Further, ${\mcl P}^{-1}:X\rightarrow X$ is self-adjoint, and ${\mcl P}^{-1}\mcl P \mbf x=\mcl P\mcl P^{-1} \mbf x=\mbf x$ for any $\mbf x  \in X := Z_{m,n}$.
	\end{lem}
	
	\begin{proof}
		This can be obtained from Theorem 3 in \cite{Miao} when we set $r=1$.
	\end{proof}
\subsection{Constructing Estimator Gains}
An analytic inverse of a generalized PI operator  $\mcl P:=\fourpi{P}{Q_1}{Q_2}{R_i}$ is an open problem. However, an exact formula is known for the inverse of $\mcl{P}$ when $R_2=R_1$, see \cite{Peet:19_estimator} and \cite{Peet:19_control}.
We find the observer gains in following Lemma.
\begin{lem}\label{lem:gains}
	Suppose PI operator $\mcl P:=\fourpi{P}{Q}{Q^T}{ R_i}$ with $R_2=R_1$ is bounded, self-adjoint and coercive. If $\mcl L = \mcl P^{-1}\mcl Z$ where $Z:=\fourpi{Z_1}{0\mkern-9.5mu/}{Z_2}{0\mkern-9.5mu/}$ and $Z_2$ is a polynomial represented as $Z_2(s)=Z^T(s)W$,
	then we get $\mcl L=\fourpi{L_1}{0\mkern-9.5mu/}{\bmat{L_{21}\\\vdots\\L_{2K}}}{0\mkern-9.5mu/}$ with $L_{1}=\hat{P} Z_{1}+\hat{H} K W$ and
	\begin{align*}
	\bmat{L_{21}\\\vdots\\L_{2K}}(s)&=\hat{R}_0(s) Z(s)^{T}\left(\hat{H}^{T} Z_{1}+W+\hat{\Gamma}KW\right)
	\end{align*} where $K$, $\Gamma$ and $\hat{H}$ are as defined in Lemma \ref{lem:inverse}.
\end{lem}

\begin{proof}
	Since $\mcl P$ is coercive, bounded, $\mcl P^{-1}=\fourpi{\hat P}{\hat Q}{\hat Q^T}{ \hat R_i}$ exists and can be obtained from Lemma \ref{lem:inverse}.
	Then follows from the formula for composition of PI operators $\mcl P^{-1}$ and $\mcl Z$ - see \cite{Sachin:19} for the formula for the composition operation.
\end{proof}	
\end{appendix}

\subsection{The equivalence between the coupled DDEs with ODE-PDE Equation and the coupled PIEs}

Consider the following coupled system dynamics,
{\small{
\begin{align}
&\begin{cases}
\dot{x}(t)=A_0x(t)+Bw(t)+\sum_{i=1}^K(A_{i}x(t-\tau_i)+B_iw(t-\tau_i))\\
z(t)=C_{1}x(t)+D_1w(t)+\sum_{i=1}^K(C_{1i}x(t-\tau_i)+D_{1i}w(t-\tau_i))\\
y(t)=C_{2}x(t)+D_2w(t)+\sum_{i=1}^K(C_{2i}x(t-\tau_i)+D_{2i}w(t-\tau_i))\end{cases}\label{eqn:DDE_ap}
\end{align}
\begin{align}
&\begin{cases}
\dot {\hat x}(t)=A_0\hat x(t)+B_v\hat v(t)+L_1(\hat y(t)-y(t))\\
\hat z(t)=C_{1}\hat x(t)+D_{1v}\hat v(t)\\
\hat y(t)=C_{2}\hat x(t)+D_{2v}\hat v(t)\\
\dot {\hat \phi}_i(t,s)=\frac{1}{\tau_i}\partial_s \hat \phi_i(t,s)+L_{2i}(s)(\hat y(t)-y(t))\\
\hat \phi(t,0)=C_{ri}\hat x(t), \
\hat v(t)=\sum_{i=1}^K C_{vi}\hat \phi_i(t,-1)\end{cases}
\label{eqn:ODE_PDE_es_ap}
\end{align}
and the coupled linear PIEs
\begin{align}
&\begin{cases}
\mcl T\dot{\mbf{x}}(t)+\mcl B_T\dot w(t)=\mcl{A}\mbf{x}(t)+\mcl{B} \omega (t)\\
z(t)=\mcl{C}_1\mbf{x}(t)+\mcl D_1 \omega(t)\\
y(t)=\mcl{C}_2\mbf{x}(t)+\mcl D_2 \omega(t)\end{cases}\label{eqn:PIE1_ap}\\
&\begin{cases}
\mcl T\dot{\hat{\mbf{x}}}(t)=\mcl{A}\hat{\mbf{x}}(t)+\mcl{L}(\hat{y}(t)-y(t))\\
\hat z(t)=\mcl{C}_1\hat{\mbf{x}}(t), \ \hat{y}(t)=\mcl{C}_2\hat{\mbf{x}}(t).\end{cases}\label{eqn:PIE2_ap}
\end{align}
These two coupled systems share the same solutions, as in the following lemma. We define the PI operators as
 	\begin{align}
 \mcl T&:=\fourpi{I}{0}{T_0}{0,0,-I},& \mcl B_T&:=\fourpi{0}{0\mkern-9.5mu/}{T_1}{0\mkern-9.5mu/}, \notag\\
 \mcl A&:=\fourpi{A_0+\sum_{i=1}^K A_i }{\tilde A}{0}{H,0,0},
 & \mcl B&:=\fourpi{B+\sum_{i=1}^KB_{i}}{0\mkern-9.5mu/}{0}{0\mkern-9.5mu/}, \notag
 \\
 \mcl C_1&:=\fourpi{C_{1}+\sum_{i=1}^KC_{1i}}{\tilde C_1}{0\mkern-9.5mu/}{0\mkern-9.5mu/},
 & \mcl C_2&:=\fourpi{C_{2}+\sum_{i=1}^KC_{2i}}{\tilde C_2}{0\mkern-9.5mu/}{0\mkern-9.5mu/}, \notag
 \\\mcl D_1&:=\fourpi{D_1+\sum_{i=1}^KD_{1i}}{0\mkern-9.5mu/}{0\mkern-9.5mu/}{0\mkern-9.5mu/},
 & \mcl D_2&:=\fourpi{D_2+\sum_{i=1}^KD_{2i}}{0\mkern-9.5mu/}{0\mkern-9.5mu/}{0\mkern-9.5mu/},\notag \\
 L&=\fourpi{L_1}{0\mkern-9.5mu/}{\bmat{L_{21}\\\vdots\\L_{2K}}}{0\mkern-9.5mu/}. \label{eqn:PIs_ap}
 \end{align}
 where

 {\small{
 		\begin{align}
 		&C_{ri}=\bmat{I\\0}, \quad B_{ri}=\bmat{0\\I},  \quad A_{ki}=\bmat{A_i&B_i}, \notag\\
 		&C_{k1i}=\bmat{C_{1i}&D_{1i}},  \quad C_{k2i}=\bmat{C_{2i}&D_{2i}},\notag\\
 		&T_0=\bmat{C_{r1}\\ \vdots\\ C_{rK}}, \quad
 		T_1=\bmat{B_{r1}\\ \vdots\\ B_{rK}}, \quad H=\text{diag}\left\{\frac{1}{\tau_1}I,\cdots,\frac{1}{\tau_K}I \right\}, \notag\\
 		&\tilde A=-\bmat{A_{k1},&\cdots,&A_{kK}},\quad
 		\tilde C_1=-\bmat{C_{k11},&\cdots,&C_{k1K}}, \notag\\
 		& \tilde C_2=-\bmat{C_{k21},&\cdots,&C_{k2K}},\label{Eqn:ACC_ap}
 		\end{align}}}
\begin{lem}
	\label{lem:Equivalence} Suppose $\mcl T, \mcl B_T, \mcl A, \mcl B, \mcl C_1, \mcl C_2, \mcl D_1, \mcl D_2, \mcl L,$ are as defined above. For given $w\in W^{1,2}[0,\infty)^r$, if $x$,  $z$, $y$, $\hat x$, $\hat z$, $\hat y$, $\hat \phi_i$ satisfy  Eq.~\eqref{eqn:DDE_ap}--\eqref{eqn:ODE_PDE_es_ap}, then $z$, $y$, $\hat z$ and $\hat y$ also satisfy the PIE defined by  \eqref{eqn:PIE1_ap}--\eqref{eqn:PIE2_ap}
	and \begin{align}
	\mbf x(t)=\bmat{x(t)\\\partial_s\phi_1(t,\cdot)\\\vdots\\\partial_s\phi_K(t,\cdot)}, \quad
	\hat{\mbf x}(t)=\bmat{\hat x(t)\\\partial_s\hat \phi_1(t,\cdot)\\\vdots\\\partial_s\hat \phi_K(t,\cdot)}, \label{Eqn:x_hatx_ap}
	\end{align}
	where $\phi_i=C_{ri}x(t+\tau_is)+B_{ri}w(t+\tau_is)$.
	Furthermore, if $\mbf x$, $\hat{\mbf x}$, $z$ and $y$ satisfy the PIE defined by Eq.\eqref{eqn:PIE1_ap}--\eqref{eqn:PIE2_ap},
	then $z$, $y$, $\hat z$, and $\hat y$ also satisfy Eq.~\eqref{eqn:DDE_ap}--\eqref{eqn:ODE_PDE_es_ap} where
	\begin{align}
	\bmat{x(t)\\\cdot}=\mcl T\mbf x(t)+\mcl B_Tw(t),
	\bmat{\hat x(t)\\\cdot}=\mcl T\hat{\mbf x}(t).
	\end{align}
\end{lem}

\vspace{0.3cm}
\begin{proof}
	For given $w\in W^{1,2}[0,\infty)^r$, suppose $x$,  $z$, $y$, $\hat x$, $\hat z$, $\hat y$, and $\hat \phi_i$ satisfy Eq.~\eqref{eqn:DDE_ap}--\eqref{eqn:ODE_PDE_es_ap}.
	Then, from Lemma~\ref{lem:ODEPDE_PIE}, $\mbf x$, $z$ $y$, $\hat x$, $\hat z$, $\hat y$, and $\hat \phi_i$ also satisfy Eq.~\eqref{eqn:PIE1_ap} and \eqref{eqn:ODE_PDE_es_ap} where
	\begin{align}
	\mbf x(t)=\bmat{x(t)\\\partial_s\phi_1(t,\cdot)\\\vdots\\\partial_s\phi_K(t,\cdot)}
	\end{align}
	and $\phi_i(t,s)=C_{ri}x(t+\tau_is)+B_{ri}w(t+\tau_is)$, and vice versa.
	
	For $\hat x$, $\hat \phi_i$ defined in Eq.~\eqref{eqn:ODE_PDE_es_ap}, define
	\begin{align}
	\hat{\mbf x}(t)=\bmat{\hat x(t)\\\partial_s\hat \phi_1(t,\cdot)\\\vdots\\\partial_s\hat \phi_K(t,\cdot)}.
	\end{align}
	Using Fundamental Theorem of Calculus and boundary conditions, we get
	\begin{align*}
	\mcl{T}\hat{\mbf x}(t) = \bmat{\hat x(t)\\\hat \phi_1(t,\cdot)\\\vdots\\\hat\phi_K(t,\cdot)},\quad \mcl{A}\hat{\mbf x}(t) = \bmat{A_0\hat x(t)+B_v\hat v(t)\\\frac{1}{\tau_1}\partial_s \hat \phi_1(t,s)\\\vdots\\\frac{1}{\tau_K}\partial_s \hat \phi_K(t,s)}
	\end{align*}
	and
	\begin{align*}
	\mcl{C}_i\hat{\mbf x}(t) = C_{i}\hat x(t)+D_{iv}\hat v(t).
	\end{align*}
	Then \begin{align*}
	\mcl T\dot{\hat{\mbf{x}}}(t)&=\mcl{A}\hat{\mbf{x}}(t)+\bmat{L_1(\hat{y}(t)-y(t))\\\bmat{L_{21}\\\vdots\\L_{2K}}(\hat{y}(t)-y(t))}\\
	\hat z(t)&=\mcl{C}_1\hat{\mbf{x}}(t),  \hat{y}(t)=\mcl{C}_2\hat{\mbf{x}}(t).
	\end{align*}
	Finally, using PI notation for the observer gains $\mcl L$ in Eq.~\eqref{eqn:PIs_ap}, we get \eqref{eqn:PIE2_ap}. Then, any $\hat x$, $\hat \phi_i$, $\hat{z}$, $\hat{y}$, $y$ that satisfies Eq.~\eqref{eqn:ODE_PDE_es_ap}, $\mbf {\hat x}$, $\hat z$, $\hat y$, $y$  also satisfy  \eqref{eqn:PIE2_ap}, where \begin{align}
	\hat{\mbf x}(t)=\bmat{\hat x(t)\\\partial_s\hat \phi_1(t,\cdot)\\\vdots\\\partial_s\hat \phi_K(t,\cdot)},
	\end{align} and vice versa.
	
	Then, for given $w\in W^{1,2}[0,\infty)^r$, if $x$,  $z$, $y$, $\hat x$, $\hat z$, $\hat y$, and $\hat \phi_i$ satisfy Eq.~\eqref{eqn:DDE_ap}--\eqref{eqn:ODE_PDE_es_ap}, then $z$, $y$, $\hat z$ and $\hat y$ also satisfy the PIE defined by  \eqref{eqn:PIE1_ap}--\eqref{eqn:PIE2_ap} for $\mbf{x}$ and $\hat{\mbf{x}}$ as defined in Eq.~\eqref{Eqn:x_hatx_ap}.
\end{proof}	


\vspace{-0.4cm}
\begin{thebibliography}{99}     

  \bibitem{Smith:57} O. J. M. Smith, Closer control of loops with dead time, Chemical Engineering Progress, vol. 53, no. 5, 217–219, 1957.

  \bibitem{Zhou:13} B. Zhou,  Z. Y. Li, and Z. L. Lin, Observer based output feedback control of linear systems with input and output delays. Automatica, vol. 49, no. 7, 2039--2052, 2013.

  \bibitem{Wang:19} J. Wang, Y. Pi, Y. Hu, and Z. Zhu,  State-observer design of a PDE-modeled mining cable elevator with time-varying sensor delays. IEEE Transactions on Control Systems Technology, 2019, doi:10.1109/tcst.2019.2897077.


 \bibitem{Luenberger:71} D. Luenberger, An introduction to observers. IEEE Transactions on Automatic Control, vol. 16, no. 6, 596--602, 1971.

 \bibitem{Kalman:95} G. Welch, and G. Bishop, An introduction to the Kalman filter. 41--95, 1995.


  \bibitem{Fridman:01} E. Fridman, and S. Uri, A new H$_\infty$ filter design for linear time delay systems. IEEE Transactions on Signal Processing, vol. 49, no. 11, 2839--2843, 2001.

 \bibitem{Fattouh:00} A. Fattouh, O. Sename, and J. M. Dion, $H_\infty$ controller and observer design for linear systems with point and distributed time-delays. IFAC Proceedings Volumes, vol. 33, no. 23, 247--252, 2000.

 \bibitem{Peet:19_estimator} M. M. Peet, $H_\infty$-Optimal estimation of systems with multiple state delays: Part 2, American Control Conference, 2019.


 \bibitem{Krstic:08} M. Krstic and A. Smyshlyaev, Backstepping boundary control for first-order hyperbolic PDEs and application to systems with actuator and sensor delays. Systems and Control Letters, vol. 57, no. 9, 750--758, 2008.

 \bibitem{Ali:18} T. Ahmed-Ali, F. Giri, M. Krstic, M.  Kahelras, PDE based observer design for nonlinear systems with large output delay. Systems and Control Letters,  vol. 113, 1--8, 2018.





 \bibitem{Peet:14} M. Peet, LMI parametrization of Lyapunov functions for infinite-dimensional systems: A framework, American Control Conference, 359--366, 2014.

 \bibitem{Peet:19_control}  M. M. Peet, $H_\infty$-Optimal control of systems with multiple state delays: Part 1, American Control Conference, 2019.







\bibitem{Hale} J. Hale, Functional differential equations, in Analytic
theory of differential equations. Springer, 9--22, 1971.

\bibitem{Peet:19_tds}  M. M. Peet,  Representation of Networks and Systems with Delay: DDEs, DDFs, ODE-PDEs and PIEs. arXiv preprint arXiv:1910.03881, 2019.


\bibitem{Sachin:19} S. Sachin, A. Das, and M. M. Peet. PIETOOLS: A matlab toolbox for manipulation and optimization of Partial Intergral operators, American Control Conference, in press, 2019.



\bibitem{Miao} G. Miao, M. M. Peet, and K. Gu, Inversion of separable kernel operators in coupled differential-functional equations and application to controller synthesis. IFAC-PapersOnLine, vol. 50, no. 1, 6513-6518.
%
\bibitem{de} C. E. de Souza, R. M. Palhares, and P. D. Peres, Robust H$_\infty$ design for uncertain linear systems with multiple time-varying state delays. IEEE Transactions on Signal Processing, vol. 49, no. 3, 569-576, 2001.


%

\end{thebibliography}
\end{document}